\documentclass[11pt]{article}

\usepackage{blindtext}
\usepackage{amssymb,latexsym,amsmath,amsthm}     
\usepackage{amsfonts}
\usepackage{float}
\usepackage{hyperref}
\usepackage{enumerate}
\usepackage{color}
\usepackage{bm}
\usepackage[font=itshape]{quoting}
\usepackage{indentfirst}
\usepackage[shortlabels]{enumitem}
\usepackage{caption}
\usepackage{hypcap}
\usepackage{breqn}
\usepackage[bottom]{footmisc}

\newcommand{\R}{\mathcal{R}}
\newcommand{\Rbold}{\bm{\mathcal{R}}}

\newcommand{\Q}{\mathcal{Q}}

\renewcommand{\P}{\mathcal{P}}

\newcommand{\set}[1]{\left\{#1\right\}}

\newcommand{\abs}[1]{\left|#1\right|}
\newcommand{\norm}[1]{\left\lVert #1\right\rVert}

\newcommand{\paren}[1]{\left(#1\right)}
\newcommand{\ds}[1]{\displaystyle}

\newcommand{\balpha}{\bm{\alpha}}
\newcommand{\bbeta}{\bm{\beta}}

\newcommand{\bnu}{\bm{\nu}}

\makeatletter
\providecommand*{\cupdot}{%
  \mathbin{%
    \mathpalette\@cupdot{}%
  }%
}
\newcommand*{\@cupdot}[2]{%
  \ooalign{%
    $\m@th#1\cup$\cr
    \hidewidth$\m@th#1\cdot$\hidewidth
  }%
}
\makeatother

\DeclareMathOperator{\E}{E}
\DeclareMathOperator{\OPT}{OPT}
\DeclareMathOperator{\V}{V}
\DeclareMathOperator{\OBJ}{OBJ}
\DeclareMathOperator{\FEAS}{FEAS}
\DeclareMathOperator{\supp}{supp}
\DeclareMathOperator{\SUPP}{SUPP}

\theoremstyle{plain}
\newtheorem{theorem}{Theorem}
\newtheorem{proposition}[theorem]{Proposition}

\newtheorem{corollary}{Corollary}
\newtheorem{conj}{Conjecture}

\newtheorem{obs}{Observation}
\newtheorem*{prob2}{Optimization Problem (OPT)}
\newtheorem*{prob3}{Optimization Problem (OPT 2)}
\theoremstyle{notation}

\newtheorem{lemma}{Lemma}
\newtheorem{claim}{Claim}
\theoremstyle{definition}

\theoremstyle{remark}

\theoremstyle{question}

\theoremstyle{construction}
\newtheorem*{const1}{Construction ($\bm{G_{\balpha}(n)}$)}
\newtheorem*{const2}{Construction ($\bm{\text{SUPP}_q(\balpha)}$)}
\theoremstyle{definition}

\title{The extremality of $2$-partite Tur\'{a}n graphs with respect to the number of colorings}
\author{Melissa M Fuentes \footnote{This document is the results of the author's doctoral dissertation. This research was supported in part by the Graduate Scholar's Award Fellowship from the University of Delaware and also in part by funds provided by the National Science Foundation Grant: 1855723.}\\ Department of Mathematics and Statistics\\ Villanova University\\ Villanova, PA 19085}
\date{Submitted: July 26, 2022}

\begin{document}

\maketitle

\begin{center}  Abstract   \end{center}
We consider a problem proposed by Linial and Wilf to determine the structure of graphs that allows the maximum number of $q$-colorings among graphs with $n$ vertices and $m$ edges. Let $T_r(n)$ denote the Tur\'{a}n graph - the complete $r$-partite graph on $n$ vertices with partition sizes as equal as possible. We prove that for all odd integers $q\geq 5$ and sufficiently large $n$, the Tur\'{a}n graph $T_2(n)$ has at least as many $q$-colorings as any other graph $G$ with the same number of vertices and edges as $T_2(n)$, with equality holding if and only if $G=T_2(n)$. Our proof builds on methods by Norine and by Loh, Pikhurko, and Sudakov, which reduces the problem to a quadratic program.

\flushleft

\setlength\parindent{20 pt}

\section{Introduction}\label{one}
For a positive integer $q$, let $[q] = \{1, 2, \ldots , q\}$. A function $f : V (G) \to [q]$ such that $f(x)\neq f(y)$ for every edge $xy$ of a graph $G$ is called a \textit{proper vertex coloring of $G$ in at most $q$ colors}, or simply a \textit{$q$-coloring} of $G$. The set $[q]$ is often referred to as the \textit{set of colors}. 

Let $P_G(q)$ denote the number of all $q$-colorings of a given graph $G$. This number was introduced and studied by Birkhoff \cite{Birkhoff}, who proved that it is always a polynomial in $q$, and is known as the \textit{chromatic polynomial} of $G$. The parameter $P_G(q)$ has been extensively studied over the past century. In particular, Linial \cite{Linial} and Wilf \cite{Wilf1, Wilf3, Wilf2} had independently posed the problem of determining the graphs that maximize the number of $q$-colorings among all graphs with $n$ vertices and $m$ edges. The problem was completely solved for $q=2$ by Lazebnik in \cite{FL1}, but remains largely open in general and has been a topic of extensive research. For a survey of numerous results on this problem, see Lazebnik \cite{FL6}. 

For a positive integer $r$, let $T_r(n)$ denote the \textit{$r$-partite Tur\'{a}n graph}, that is, the complete $r$-partite graph of order $n$ with all parts nearly equal in size (each part is of size $\lfloor n/r \rfloor$ or $\lceil n/r \rceil$). Let $t_r(n)$ denote the number of edges of $T_r(n)$. Our main motivation is the following conjecture made by Lazebnik in 1987, although it first appeared in print in \cite{FL4}. 

\begin{conj}[\cite{FL4}]\label{main_conj}
For all $n \geq r \geq 2$ and $q \geq r$, the Tur\'{a}n graph $T_r(n)$ is the only graph on $n$ vertices and $t_r(n)$ edges that attains the maximum number of $q$-colorings.
\end{conj}

When $q=r$, the statement follows from the celebrated Tur\'{a}n Theorem \cite{Turan}, since any graph with $n$ vertices and $t_r(n)$ edges that is not $T_r(n)$ does not have a $r$-coloring. Lazebnik proved Conjecture \ref{main_conj} when $r=2$ and $q\geq (n/2)^5$ in \cite{FL1} and when $n$ is a positive integer divisible by $r$ and $q \geq 2\binom{t_r(n)}{3}$ in \cite{FL3}. For $r=2$ and $q=3$, Lazebnik, Pikhurko, and Woldar \cite{FL4} proved the conjecture when $n$ is even, as well as an asymptotic version when $q=4$ for even $n$, as long as $n$ is sufficiently large. Their result for $q=4$ was extended by Tofts \cite{Tofts} to all $n \geq 4$. Loh, Pikhurko, and Sudakov proved the conjecture for $q=r+1$ for large $n$ in \cite{Loh}, and their result was later extended to all $n \geq r$ by Lazebnik and Tofts in \cite{FL5}. This was greatly improved by Norine \cite{Norine}, who developed further powerful techniques from \cite{Loh}, and showed that for any positive integers $q$ and $r$ such that $2 \leq r <q$ and $r$ divides $q$, Conjecture \ref{main_conj} is true, provided that $n$ is sufficiently large. The most recent result was by Ma and Naves \cite{Ma}, who showed that Conjecture \ref{main_conj} is true for all $q\geq 100 r^2/(\log(r))$, for $n$ is sufficiently large.

Conjecture \ref{main_conj} was widely believed to be true, especially since there are many results confirming it in several cases. However, Ma and Naves \cite{Ma} constructed counterexamples in some ranges of $r$ and $q$. For example, if $r+3\leq q \leq 2r-7$ and $r \geq 10$ then Conjecture \ref{main_conj} is false. Also, for all integers $r \geq 50000$ and $q_0$ such that $20r \leq q_0 \leq \frac{r^2}{200\log(r)}$, there exists an integer $q$ within distance at most $r$ from $q_0$, such that Conjecture \ref{main_conj} is false for $r$ and $q$. Nevertheless, Conjecture \ref{main_conj} is still believed to be true for integers $r$ and $q$, where $2\leq r \leq 9$ and $q \geq r$. The first case for which there are no explicit or asymptotic results is for $r=2$ and odd $q\geq 5$. This motivated the research done in this paper. The main result is the following theorem. 
\begin{theorem}\label{main_theorem} 
Let $q \geq 5$ be an odd integer. Then for all sufficiently large $n$, the Tur\'{a}n graph $T_2(n)$ has more $q$-colorings than any other graph with the same number of vertices and edges. 
\end{theorem}

\subsection{Notation}

All graphs in this article are finite, undirected, and have neither loops nor multiple edges. For all missing definitions and basic facts which are mentioned but not proved, we refer the reader to Bollob\'{a}s \cite{Bela}. 

For a graph $G$, let $V = V (G)$ and $E = E(G)$ denote the vertex set of $G$ and the edge set of $G$, respectively. Let $|A|$ denote the size of a set $A$. Let $e(G)=|E(G)|$ denote the number of edges of $G$. For $A \subseteq V (G)$, let $G[A]$ denote the subgraph of $G$ \textit{induced by $A$}, which means that $V(G[A]) = A$, and $E(G[A])$ consists of all edges $xy$ of $G$ with both $x$ and $y$ in $A$. For a vertex $v$ of $G$, let $d_A(v)$ denote the \textit{degree} of $v$ in $A$, the number of vertices in $A$ that are adjacent to $v$. For two disjoint subsets $A, B \subseteq V (G)$, by $G[A, B]$ we denote the bipartite subgraph of $G$ induced by $A$ and $B$, which means that $V (G[A, B]) = A \cup B$, and $E(G[A, B])$ consists of all edges of $G$ with one end-vertex in $A$ and the other in $B$. Let $k$ be a positive integer. A $k$-\textit{partition} of a set $S$ is a collection of disjoint subsets $A_1, A_2, \ldots, A_k$ (possibly empty) such that $S=A_1\cup A_2\cup \ldots \cup A_k$.

\subsection{Organization}

In section \ref{two}, the aforementioned approach from \cite{Loh} is presented, as well as two important graph constructions and powerful related results from \cite{Loh} and \cite{Ma} that will be used to prove Theorem \ref{main_theorem} are presented. In section \ref{three}, new techniques are used to solve the relevant instances of the linear optimization problem by Loh et al. In section \ref{four}, an approximate version of Theorem \ref{main_theorem} is proved.  The main result, Theorem \ref{main_theorem}, is derived in section \ref{five}, using the results of the previous sections. Finally, several open problems for future investigation are mentioned in section \ref{six}. 


\section{The linear optimization problem and associated graph constructions}\label{two}

In their breakthrough paper, Loh, Pikhurko, and Sudakov \cite{Loh} developed the optimization problem \textbf{OPT} for future researchers to use to determine the graphs that maximize the number of $q$-colorings among all graphs with the same numbers of vertices and edges. They remark that ``the remaining challenge is to find analytic arguments which solve the optimization problem for general $q$." We will solve particular cases of the optimization problem in Section \ref{four}. 

\subsection{The linear optimization problem by Loh, Pikhurko, and Sudakov}\label{two_one}

It is shown in \cite{Loh} that solving the problem proposed by Linial \cite{Linial} and Wilf \cite{Wilf1, Wilf3, Wilf2} for large $n$ reduces to a quadratically constrained linear program, which we now define. 

Let $\mathbb{R}$ denote the set of real numbers. Let $\balpha=(\alpha_A)_{A\subseteq [q], A\neq \emptyset}$ be a vector with $2^q-1$ components $\alpha_A \in \mathbb{R}$ that are indexed by the nonempty subsets $A \subseteq [q]$. When all components of $\balpha$ are nonnegative we will write $\balpha \geq 0$. The logarithms below and in the rest of the paper are natural.

Fix an integer $q\geq 2$ and a real number $\gamma$ that satisfies $0<\gamma \leq (q-1)/(2q)$. Consider the following functions of $\balpha$:
$$  \OBJ_q(\balpha):= \sum_{A\neq \emptyset} \alpha_A \log |A|; \quad \V_q(\balpha):=\sum_{A\neq \emptyset} \alpha_A; \quad \E_q(\balpha):= \sum_{\substack{\{A,B\}: A\cap B = \emptyset\\ A\neq \emptyset, B\neq \emptyset}} \alpha_A\alpha_B.   $$
The sums in $\OBJ_q(\balpha)$ and $V_q(\balpha)$ run over all nonempty subsets of $[q]$, and the sum in $\E_q(\balpha)$ runs over all \textit{unordered} pairs of disjoint nonempty subsets of $[q]$. In the remainder of this paper, we will suppress mentioning that the sets over which the sums above are taken are nonempty. 

Let $\FEAS_q(\gamma)$ be defined by 
$$ \FEAS_q(\gamma):=\{\balpha \in \mathbb{R}^{2^q-1}: \balpha \geq 0, \, \V_q(\balpha) =1, \, \text{and} \, \E_q(\balpha) \geq \gamma\} .$$
The elements of $\FEAS_q(\gamma)$ will be referred to as \textit{feasible vectors}. Our goal is to maximize $\OBJ_q(\balpha)$ over $\FEAS_q(\gamma)$.

\begin{prob2} \label{OptProb}
Find
$$ \OPT_q(\gamma):= \displaystyle \max_{\balpha \in \FEAS_q(\gamma)} \OBJ_q(\balpha). $$
\end{prob2}
As noted in \cite{Loh}, a solution of \textbf{OPT} exists by continuity of $\OPT_q(\gamma)$ and by compactness of the set $\FEAS_q(\gamma)$. We say that $\balpha$ \textit{solves} $\OPT_q(\gamma)$ (or just \textbf{OPT}) if $\balpha \in \FEAS_q(\gamma)$ and $\OBJ_q(\balpha)= \OPT_q(\gamma)$. 

Our objective in solving \textbf{OPT} is to obtain the approximate structure of a graph on $n$ vertices and at least $\gamma n^2$ edges that has the most number of $q$-colorings, provided that $n$ is sufficiently large. We define such a graph in the next section.

Loh et al. \cite{Loh} had solved \textbf{OPT} in for all $q\geq 3$ when $\gamma$ satisfies $0 \leq \gamma \leq \kappa_q$, where
$$
\kappa_q:=\paren{\sqrt{\frac{\log(q/(q-1))}{\log(q)}}+\sqrt{\frac{\log(q)}{\log(q/(q-1))}}}^{-2}\approx \frac{1}{q\log(q)}.
$$

Norine \cite{Norine} presented an argument that solves \textbf{OPT} when $\gamma = (r-1)/(2r)$, where $r$ is the number of parts in the Tur\'{an} graph, $T_r(n)$, and $q$ is divisible by $r$. In particular, Norine completely solved \textbf{OPT} for $\gamma=1/4$ (that is, $r=2$) and all even integers $q\geq 2$. In Section \ref{four} we extend his solution to all $\gamma$ within a closed interval of real numbers $[1/4 -\epsilon, 1/4]$ when $\epsilon>0$ is sufficiently small and $q\geq 5$.

\subsection{Graph constructions based on feasible vectors}\label{two_two}

The following is a graph construction from \cite{Loh} that is based on an arbitrary feasible vector $\balpha$. Note that this construction may not result in a unique graph.

\begin{const1}\label{Cons1} 
Let $\gamma$ be a real number, $n\geq 1$, and $\balpha \in \FEAS_q(\gamma)$. The $n$-vertex graph $G_{\balpha}(n)$ is constructed by partitioning the $n$ vertices into disjoint sets $V_A$, indexed by the non-empty subsets $A \subseteq [q]$, such that each $|V_A|$ differs from $n\alpha_A$ by less than $1$. For every pair of non-empty $A, B \subseteq [q]$ with $A\cap B =\emptyset$ join every vertex in $V_A$ to every vertex of $V_B$ by an edge.
\end{const1}


If all $n\alpha_A$ happened to be integers, the graph $G_{\balpha}(n)$ would be unique and $G_{\balpha}(n)$ would have precisely $\E_q(\balpha)n^2$ edges and at least $$\prod_{A\neq \emptyset} |A|^{n\alpha_A}=e^{\OBJ_q(\balpha)n}$$ $q$-colorings, since any coloring in which the vertices in $V_A$ are only colored with colors in $A$ results in a $q$-coloring of $G_{\balpha}(n)$. However, if this is not the case, the choice in size of each $V_A$ may result in a graph with fewer than $\gamma n^2$ edges (recall that $\E_q(\balpha)n^2\geq \gamma n^2$). Fortunately, we may use the Proposition \ref{LPS1} below in order to obtain another graph on $n$ vertices which is ``close" (to be defined next) to $G_{\balpha}(n)$, and which has at least $\gamma n^2$ edges. 

Given graphs $G$ and $H$ with the same set of vertices, their \textit{edit distance} is the minimum number of edges that need to be added or deleted from one graph to obtain a graph isomorphic to the other. 

\begin{proposition}[\cite{Loh}]\label{LPS1}
For any feasible vector $\balpha$, the number of edges in any graph $G_{\balpha}(n)$ differs from $\E_q(\balpha)n^2$ by less than $2^qn$. Also, for any other feasible vector $\bnu$, the edit distance between $G_{\balpha}(n)$ and $G_{\bnu}(n)$ is at most $\norm{\balpha - \bnu}_1 n^2+2^{q+1}n$, where $\norm{\cdot}_1$ is the $L^1$-norm on $\mathbb{R}^{2^q-1}$.
\end{proposition}

Proposition \ref{LPS1} and Theorem \ref{LPS2} (shown below) will be important in the proof of our main result, Theorem \ref{main_theorem}, in Section \ref{six}.

\begin{theorem}[\cite{Loh}] \label{LPS2}
For any $\delta, \kappa >0$, the following holds for all sufficiently large $n$. Let $G$ be an $n$-vertex graph with $m$ edges, where $m \leq \kappa n^2$, which has at least as many $q$-colorings as any other graph with the same number of vertices and edges. Then $G$ is $\epsilon n^2$-close to a graph $G_{\balpha}(n)$ for some feasible vector $\balpha$ which solves $\OPT_q(\gamma)$ for some $\gamma$, where $|\gamma-m/n^2|<\epsilon$ and $\gamma \leq \kappa$.
\end{theorem}

The \textit{support} of a feasible vector $\balpha=(\alpha_A)_{A\subseteq [q], A \neq \emptyset}$, denoted by $\supp_q(\balpha)$, is the collection of sets $A\subseteq [q]$ such that $\alpha_A > 0$. The graph construction below is from \cite{Ma}.
\begin{const2}
Let $\balpha \in \FEAS_q(\gamma)$. The set of vertices of the graph $\SUPP_q(\balpha)$ is $\supp_q(\balpha)$ and the edge set is formed by connecting pairs of disjoint sets. 
\end{const2}
The graph $\SUPP_q(\balpha)$ is called the \textit{support graph} of a feasible vector $\balpha$. We define two classes of graphs $\SUPP_q(\balpha)$ as follows. Let $\P_k$ be the class of all graphs $\SUPP_q(\balpha)$ for which $\supp_q(\balpha)$ forms a $k$-partition $A_1, A_2, \ldots, A_k$ of $[q]$, with only nonempty sets. Let $\Q_k$ be the class of all graphs $\SUPP_q(\balpha)$ for which $\supp_q(\balpha)$ consists of a $k$-partition $A_1, A_2, \ldots, A_k$ of $[q]$, with only nonempty sets, together with the set $A_1\cup A_2$. 

Fix a positive integer $k$ and let $\balpha$ be a vector such that the graph $\SUPP_q(\balpha)$ is in $\P_k$ or $\Q_k$. Consider a restricted version of \textbf{OPT} for such vectors $\balpha$:
\begin{prob3} \label{OptProb2}
Maximize
$$ 
\displaystyle \sum_{i=1}^k \alpha_{A_i} \log(|A_i|)+\alpha_{A_1\cup A_2} \log(|A_1|+|A_2|), 
$$
subject to
\[
\sum_{i=1}^k \alpha_{A_i} +\alpha_{A_1\cup A_2}=1, \quad \sum_{i=1}^k \alpha_{A_i}^2 +\alpha_{A_1\cup A_2}^2+2\alpha_{A_1\cup A_2}(\alpha_{A_1}+\alpha_{A_2})\leq 1-2\gamma,
\]
and
\[
\alpha_{A_i}\geq 0, \alpha_{A_1\cup A_2}\geq 0, \text{ for } i=1, \ldots, k.
\]
\end{prob3}
The conditions of \textbf{OPT 2} are consistent with the conditions of \textbf{OPT}, when restricted to vectors with support in $\Q_k \cup \P_k$. The following observation was made in (\cite{Ma}) using the Cauchy-Schwarz inequality.

\begin{obs}\label{obs}
We have $k\geq 1/\lceil 1-2\gamma \rceil$ in \textbf{OPT 2}, with a strict inequality holding if $\alpha_{A_1}+\alpha_{A_2}>0$ and $\alpha_{A_1\cup A_2}>0$.
\end{obs}

Define 
$$
\P=\bigcup_{\lceil \frac{1}{1-2\gamma} \rceil \leq k \leq q} \P_k.
$$
The next theorem determines the structure of the graph $\SUPP_q(\balpha)$ when $\balpha$ solves \textbf{OPT}. 

\begin{theorem}[\cite{Ma}] \label{JackTheRipper2}
For an integer $q$ and real $\gamma$ that satisfy $0< \gamma \leq (q-1)/(2q)$, all solutions to \textbf{OPT} are such that $\SUPP_q(\balpha)$ is in either $\P$ or $\Q_{\lceil 1/(1-2\gamma) \rceil}$. When $\gamma < (q-1)/(2q)$, we have $\SUPP_q(\balpha) \notin \P_q$.
\end{theorem}


\section{Relevant solutions of \textbf{OPT}}\label{three}

In this section we use new methods to provide an analytic solution to \textbf{OPT} for odd $q\geq 5$ for all $\gamma$ sufficiently close to $1/4$. Although a computer solution can be found for fixed $\gamma$, an analytic solution is necessary to solve \textbf{OPT} for all $\gamma$ within a closed interval of real numbers $[1/4 -\epsilon, 1/4]$ for small $\epsilon >0$. 

The main result of this section is the following theorem. For any set $S$, we denote its \textit{complement} by $S^c$.

\begin{theorem} \label{Opt5}
The following holds for all $\gamma$ sufficiently close to $1/4$. Any solution $\balpha$ of {\rm\textbf{OPT}} for odd $q\geq 5$ has $\supp_q(\balpha)=\{A, A^c\},$ where $|A|=\lceil q/2\rceil$, and
$$ \alpha_{A}=\frac{1+\sqrt{1-4\gamma}}{2} \quad \text{and} \quad \alpha_{A^c}=1-\alpha_{A},$$
which gives
$$
\OPT_q(\gamma)=\frac{1}{2}\log\paren{\lceil q/2 \rceil \cdot \lfloor q/2 \rfloor}+\frac{\sqrt{1-4\gamma}}{2}\log\paren{\frac{\lceil q/2 \rceil}{\lfloor q/2 \rfloor}}.
$$
\end{theorem}

Note that the vector $\balpha$ defined in Theorem \ref{Opt5} is a feasible vector. We will use a constraint that is equivalent to $\E_q(\bbeta)\geq \gamma$ for all $\bbeta \in \FEAS(\gamma)$, specifically, $1-2\E_q(\bbeta) \leq 1-2\gamma$. In addition, note that 
\begin{equation} \label{Econstraint}
1-2\E_q(\bbeta) = \sum_{(B,S): B\cap S\neq \emptyset} \beta_B\beta_S =\sum_S P_S(q)\beta_S,
\end{equation}
where
$$
P_S(q):=\sum_{B: B\cap S \neq \emptyset} \beta_B.
$$
Therefore,
\begin{equation}\label{PSineq}
\sum_S P_S(q)\beta_S = 1-2\E_q(\bbeta) \leq  1-2\gamma.
\end{equation}

Recall that $q$ is the number of colors. We continue with two lemmas  which hold for all integers $q\geq 2$ that will be used in the proof of Theorem \ref{Opt5}. For $\bbeta=(\beta_S)_{S\subseteq [q], S\neq \emptyset} \in \FEAS_q(\gamma)$ and every $S \in \supp_q(\bbeta)$, let
$$ 
Q_S(q):=\frac{|S|}{q}.
$$

Lemma \ref{LemmaTech1} below is a technical result to be used in the proof of Lemma \ref{Claim2}. 

\begin{lemma}\label{LemmaTech1}
Let $\gamma$ satisfy $0 < \gamma \leq 1/4$ and let $\bbeta \in \FEAS_q(\gamma)$. Then 
\begin{equation}\label{ineq1}
\sum_{S} \paren{2\sqrt{2Q_S(q)}-(3-4\gamma)} \beta_S \leq 0.
\end{equation}
\end{lemma}

\begin{proof}
By a comparison of the geometric and arithmetic mean of two numbers (AMGM) we obtain
\begin{equation}\label{amgm}
2\sqrt{2Q_S(q)} \leq (2-4\gamma)\frac{Q_S(q)}{P_S(q)}+\frac{1}{1-2\gamma}P_S(q).
\end{equation}
Thus,
\begin{equation} \label{elprimero} 
\sum_{S} \paren{2\sqrt{2Q_S(q)}-(3-4\gamma)} \beta_S \leq \sum_{S} \paren{(2-4\gamma)\frac{Q_S(q)}{P_S(q)}+\frac{1}{1-2\gamma}P_S(q)-(3-4\gamma)} \beta_S.
\end{equation}

Now we show that
\begin{equation} \label{elsegundo}
\sum_{S} \paren{(2-4\gamma)\frac{Q_S(q)}{P_S(q)}+\frac{1}{1-2\gamma}P_S(q)-(3-4\gamma)} \beta_S \leq 0.
\end{equation} 
Let us prove that
\begin{equation} \label{2ndineq}
\sum_{S} \frac{Q_S(q)}{P_S(q)} \beta_S\leq 1.
\end{equation}
Let $\emptyset \neq S \subseteq [q]$, and $\chi_S:[q] \to \set{0,1}$ be defined by 
$$ \chi_S(x)= \begin{cases}
		      	1  & \text{ if } x \in S\\
		      	0 & \text{ if } x \notin S.
			  \end{cases}
$$
Take any $x \in [q]$. Then
\begin{equation}\label{chiineq}
\displaystyle
\sum_S \frac{\chi_S(x)}{P_S(q)} \beta_S = \sum_{S: x \in S} \frac{1}{P_S(q)} \beta_S = \sum_{S: x \in S} \paren{\frac{1}{\displaystyle \sum_{B:B\cap S \neq \emptyset} \beta_B}} \beta_S \leq \sum_{S: x \in S} \paren{\frac{1}{\displaystyle \sum_{B:x \in B} \beta_B}} \beta_S=1,
\end{equation}
where for $x \in S$, the last inequality in (\ref{chiineq}) is due to the fact that $x \in B$ implies $B \cap S \neq \emptyset$. 

Note that $\sum_{x \in [q]} \chi_S(x) =\abs{S}$. Then, using (\ref{chiineq}), we have
\begin{equation} 
q=\sum_{x\in [q]} 1 \geq \sum_{x \in [q]} \sum_S \frac{\chi_S(x)}{P_S(q)} \beta_S  = \sum_S \frac{\sum_{x \in [q]} \chi_S(x)}{P_S(q)} \beta_S =\sum_S \frac{|S|}{P_S(q)} \beta_S  =q\sum_S \frac{Q_S(q)}{P_S(q)}\beta_S \nonumber,
\end{equation}
which implies (\ref{2ndineq}).

Using (\ref{2ndineq}), (\ref{PSineq}), and $\V_q(\bbeta)=1$, we obtain (\ref{elsegundo}). Therefore, (\ref{elprimero}) and (\ref{elsegundo}) imply that (\ref{ineq1}) holds. The proof is complete.
\end{proof}

Next we show that the sum of the components corresponding to sets of size roughly $q/2$ in the support of a feasible vector $\bbeta$ with $\OBJ_q(\bbeta)$ being at least as large as the optimal value claimed in Theorem \ref{Opt5} carry more ``weight'' than the sum of components corresponding to sets of other sizes (recall that $\V_q(\bbeta)=1$ and all components of feasible vectors are nonnegative).


\begin{lemma}\label{Claim2}
Let $\bbeta \in \FEAS_q(\gamma)$, where $q \geq 5$. If
\begin{equation} \label{OBJineq}
\OBJ_q(\bbeta) \geq \OBJ_q(\balpha)= \frac{1}{2}\log\paren{\lceil q/2 \rceil \cdot \lfloor q/2 \rfloor}+\frac{\sqrt{1-4\gamma}}{2}\log\paren{\frac{\lceil q/2 \rceil}{\lfloor q/2 \rfloor}},
\end{equation}
then for all $\gamma$ sufficiently close to $1/4$,
\begin{equation} \label{betaineq0}
\sum_{S: |S| =\lceil q/2 \rceil,\lfloor q/2 \rfloor} \beta_S \geq \begin{cases} 
0.76 & \quad \text{if $q=5$}\\
0.80  & \quad \text{if $q\geq 7$}.
\end{cases}.
\end{equation}
\end{lemma}

\begin{proof}
Let $\bbeta=(\beta_S) \in \FEAS_q(\gamma)$. Then by Lemma \ref{LemmaTech1}, we have
$$  \sum_{S} \paren{2\sqrt{2Q_S(q)}-(3-4\gamma)} \beta_S \leq 0.$$
Then
\begin{align*}
&\sum_{S} \paren{2\sqrt{2Q_S(q)}-(3-4\gamma)} \beta_S \\
&\leq \OBJ_q(\bbeta)-\frac{1}{2}\log\paren{\lceil q/2 \rceil \cdot \lfloor q/2 \rfloor}-\frac{\sqrt{1-4\gamma}}{2}\log\paren{\frac{\lceil q/2 \rceil}{\lfloor q/2 \rfloor}} \quad \text{ by ($\ref{OBJineq}$),}\\
&= \sum_S \log\paren{|S|}\beta_S-\frac{1}{2}\log\paren{\lceil q/2 \rceil \cdot \lfloor q/2 \rfloor}+\frac{\sqrt{1-4\gamma}}{2}\log\paren{\frac{\lfloor q/2 \rfloor}{\lceil q/2 \rceil}} \\
& =\sum_S \paren{\log(2Q_S(q))+\frac{1}{2}\log\paren{\frac{q^2}{q^2-1}}+\frac{\sqrt{1-4\gamma}}{2}\log \paren{\frac{\lfloor q/2 \rfloor}{\lceil q/2 \rceil}}} \beta_S,
\end{align*}
and therefore,
\begin{align} \label{mainineq2} \small
0&\leq \sum_S \paren{\log(2Q_S(q))+\frac{1}{2}\log\paren{\frac{q^2}{q^2-1}}+\frac{\sqrt{1-4\gamma}}{2}\log \paren{\frac{\lfloor q/2 \rfloor}{\lceil q/2 \rceil}} -  2\sqrt{2Q_S(q)}+3-4\gamma)} \beta_S \nonumber \\
&=\sum_{m \in \{2,4,6,\dots, 2q\}} \,\sum_{S: |S|=m/2}  f(m, q, \gamma) \, \beta_S,
\end{align}
\normalsize
where we define
\begin{equation}\label{fbetaineq}\small 
f(m,q,\gamma)=\log \paren{\frac{m}{q}} +\frac{1}{2}\log \paren{\frac{q^2}{q^2-1}}+\frac{\sqrt{1-4\gamma}}{2}\log \paren{\frac{\lfloor q/2 \rfloor}{\lceil q/2 \rceil}}-2\sqrt{\frac{m}{q}}+3-4\gamma.
\end{equation}
\normalsize

The following claim will help us find maximum values of $f(m, q, \gamma)$ over certain ranges of $m$.

\begin{claim} \label{Claim0}
Let $q\geq 5$ and $m \in [2, 2q]$ be real numbers. For fixed $q$, define
$$ h_q(m)=\log\paren{\frac{m}{q}} -2\sqrt{\frac{m}{q}}.$$
Then $h_q(m)$ is increasing on $[2,q)$ and is decreasing on $(q,2q]$. Moreover, for a real number $k\in [1,q-2]$, 
\begin{equation}\label{Claim0ineqmain}  
h_q(q-k)<h_q(q+k).
\end{equation}
\end{claim} 
\begin{proof}[Proof of Claim \ref{Claim0}]
Observe that 
$$ \frac{d}{dm} h_q(m)=\frac{1-\sqrt{m/q}}{m}.$$
Since $\sqrt{m/q}>1$ whenever $q<m\leq 2q$ and $\sqrt{m/q}<1$ whenever $2\leq m <q$, then the first part of the claim holds.

If we substitute $t=k/q$, the inequality (\ref{Claim0ineqmain}) is equivalent to
\begin{equation} \label{Claim0ineq1}
\log\left( \frac{1+t}{1-t}\right)>2\left(\sqrt{1+t}-\sqrt{1-t}\right)
\end{equation}
for $1/q \leq t \leq 1-2/q$. 

Let us show that (\ref{Claim0ineq1}) holds for any $t\in (0,1)$. Since $[1/q, 1-2/q] \subseteq (0,1)$, it will imply (\ref{Claim0ineqmain}).

Consider the function
$$
 L(t):= \log\left( \frac{1+t}{1-t}\right) -2\left(\sqrt{1+t}-\sqrt{1-t}\right) \, .
$$
For $t \in (0,1)$, we have
$$
L'(t) = \frac{2}{1-t^2} - \left( \frac{1}{\sqrt{1+t}} + \frac{1}{\sqrt{1-t}}  \right) 
= \frac{2 - \sqrt{1-t^2}\left(\sqrt{1+t} + \sqrt{1-t}\right)}{1-t^2}.
$$
It can be easily verified that
$$
\sqrt{1-t^2} \cdot (\sqrt{1+t} + \sqrt{1-t}) < \sqrt{1+t} + \sqrt{1-t} < 2.
$$
Therefore, $L'(t)>0$ for all $t \in (0,1)$. As$L(0) = 0$ and $L$ is continuous at $0$, $L(t)>0$ on $(0,1)$, which proves the inequality (\ref{Claim0ineq1}). Therefore,(\ref{Claim0ineqmain}) holds. 
\end{proof}

Let $g(q,\gamma):=\frac{1}{2}\log \paren{\frac{q^2}{q^2-1}}+\frac{\sqrt{1-4\gamma}}{2}\log \paren{\frac{\lfloor q/2 \rfloor}{\lceil q/2 \rceil}}+3-4\gamma.$
Claim \ref{Claim0} implies that 
\[
\max_{m\in \{q-1, q+1\}} f(m, q, \gamma)=g(q,\gamma)\, +\max_{m\in \{q-1, q+1\}} h_q(m) =g(q, \gamma)\,+\, h_q(q+1)=f(q+1, q, \gamma)
\]
and 
\[
\max_{m\notin \{q-1, q+1\}} f(m, q, \gamma)=g(q,\gamma)+\max_{m\notin \{q-1, q+1\}} h_q(m)=g(q, \gamma) + h_q(q+3)=f(q+3, q, \gamma).
\]
By combining these observations with (\ref{fbetaineq}) we obtain
\begin{align*}
0 &\leq \sum_{m \in \{q-1, q+1 \}} \,\sum_{S:|S|=m/2}  f(m, q, \gamma) \, \beta_S \, + \sum_{m \notin \{q-1, q+1 \}} \,\sum_{S:|S|=m/2}  f(m, q, \gamma) \, \beta_S\\
   &\leq f(q+1, q, \gamma) \sum_{S: |S| \in \{\lfloor q/2 \rfloor, \lceil q/2 \rceil \}}\beta_S+ f(q+3, q, \gamma) \sum_{S: |S|\notin \{\lfloor q/2 \rfloor, \lceil q/2 \rceil\}}\beta_S\\
&= (f(q+1, q, \gamma)-f(q+3, q, \gamma)) \sum_{S: |S|\in \{\lfloor q/2 \rfloor, \lceil q/2 \rceil\}}\beta_S \,+ f(q+3, q, \gamma).
\end{align*}
By Claim \ref{Claim0}, we have $f(q+1,q, 1/4)-f(q+3, q, 1/4)>0$ for all $q\geq 5$. Since for fixed $q$ and $m$, $f(m,q, \gamma)$ is continuous as a function of $\gamma$ on $[0, 1/4]$, then $$f(q+1, q, \gamma)-f(q+3, q, \gamma)>0$$ for all $\gamma$ sufficiently close to $1/4$. Therefore,
\begin{align} \label{betainequito}
\sum_{S: |S|\in \{\lceil q/2 \rceil, \lfloor q/2 \rfloor \}} \beta_S &\geq \frac{-f(q+3, q, \gamma)}{-f(q+3, q, \gamma)+f(q+1, q, \gamma)} \nonumber \\
                                                                                             & =\frac{\log\paren{\frac{q+3}{\sqrt{q^2-1}}}-2\sqrt{\frac{q+3}{q}}-\frac{\sqrt{1-4\gamma}}{2}\log\paren{\frac{\lceil q/2 \rceil}{\lfloor q/2 \rfloor}}-4\gamma+3}{\log\paren{\frac{q+3}{q+1}}+2\paren{\sqrt{\frac{q+1}{q}}-\sqrt{\frac{q+3}{q}}}}\nonumber =:\mu(q,\gamma).
\end{align}
Let 
$$ 
A=\log\paren{\frac{q+3}{q+1}}+2\paren{\sqrt{\frac{q+1}{q}}-\sqrt{\frac{q+3}{q}}}.
$$
Then 
$$ 
\frac{d}{d\gamma} \mu(q,\gamma)=\frac{1}{A}\paren{\frac{\log\paren{\frac{\lceil q/2 \rceil}{\lfloor q/2 \rfloor}}}{\sqrt{1-4\gamma}}-4}.
$$
Since for all $q\geq 5$ we have $A<0$ and
$$
 \frac{\log\paren{\frac{\lceil q/2 \rceil}{\lfloor q/2 \rfloor}}}{\sqrt{1-4\gamma}}-4>0,
$$
for $\gamma \leq 1/4$ sufficiently close to $1/4$, then $\frac{d}{d\gamma} \mu(q,\gamma) <0$ and
\begin{equation} \label{muinequito} 
\mu(q, \gamma) \geq \mu(q, 1/4)=\frac{\log\paren{\frac{q+3}{\sqrt{q^2-1}}}-2\sqrt{\frac{q+3}{q}}+2}{\log\paren{\frac{q+3}{q+1}}+2\paren{\sqrt{\frac{q+1}{q}}-\sqrt{\frac{q+3}{q}}}}  
\end{equation}
for all $q\geq 5$ and $\gamma \leq 1/4$ sufficiently close to $1/4$.

Note that $\mu(5,1/4)\approx 0.768$. We prove the following claim to show that $\mu(q, 1/4)\geq 0.8$ for all $q\geq 7$. 
\begin{claim} \label{Claim1}
For all $q\geq 7$, $\mu(q, 1/4)\geq 4/5$.
\end{claim} 
\begin{proof}
We will prove that, for all $q\geq 7$,
\begin{equation}\label{Claim1ineq}
\frac{\log\paren{\frac{q+3}{\sqrt{q^2-1}}}-2\sqrt{\frac{q+3}{q}}+2}{\log\paren{\frac{q+3}{q+1}}+2\paren{\sqrt{\frac{q+1}{q}}-\sqrt{\frac{q+3}{q}}}} \geq \frac{4}{5}.
\end{equation}

Note that since $f(x)=x-\log(x)$ is strictly increasing for all $x>1$, then
$$
\sqrt{\frac{q+1}{q}}-\log\paren{\sqrt{\frac{q+1}{q}}} < \sqrt{\frac{q+3}{q}}-\log\paren{\sqrt{\frac{q+3}{q}}},
$$
and hence, the denominator of the left side of (\ref{Claim1ineq}) is negative for all $q\geq 7$. Therefore, proving (\ref{Claim1ineq}) is equivalent to showing that
$$
2\sqrt{\frac{q+3}{q}}-2\log \sqrt{\frac{q+3}{q}} +8\sqrt{\frac{q+1}{q}}-3\log \sqrt{\frac{q+1}{q}}-5\log\sqrt{\frac{q}{q-1}}-10\geq 0.
$$
With the substitution $q = \frac{1}{x}$, it suffices to prove that, for all $x\in [0, 1/7]$,
$$
F(x):=2\sqrt{1 + 3x} - \ln(1 + 3x) + 8\sqrt{1 + x} - \frac32\ln(1 + x)+\frac{5}{2}\ln(1 - x) -10 \geq 0 .
$$
We have
$$F'(x) = \frac{3}{\sqrt{1 + 3x}}
- \frac{3}{1 + 3x} + \frac{4}{\sqrt{1 + x}}
- \frac{3}{2(1 + x)} - \frac{5}{2(1 - x)},$$
$$F''(x) = - \frac{9}{2(1 + 3x)^{3/2}} + \frac{9}{(1 + 3x)^2} - \frac{2}{(1 + x)^{3/2}} + \frac{3}{2(1 + x)^2} - \frac{5}{2(1 - x)^2},$$
and
$$F'''(x) = \frac{81}{4(1 + 3x)^{5/2}} - \frac{54}{(1 + 3x)^3} + \frac{3}{(1 + x)^{5/2}} - \frac{3}{(1 + x)^3}
- \frac{5}{(1 - x)^3}.$$

We have, for all $x\in [0, 1/7]$,
\begin{align*}
	F'''(x)
	&\le \left(\frac{81}{4(1 + 3x)^2} -  \frac{54}{(1 + 3x)^3}\right) + \left(\frac{3}{(1 + x)^2} - \frac{3}{(1 + x)^3}
	- \frac{5}{(1 + x)^3}\right)\\
	&= \frac{27(9x - 5)}{4(1 + 3x)^3} + \frac{3x - 5}{(1 + x)^3} < 0.
\end{align*}
Note that $F''(0) > 0$ and $F''(1/7) < 0$. Thus, there exists $x_0 \in (0, 1/7)$ such that $F''(x_0) = 0$, $F''(x) > 0$ for $x\in [0, x_0)$,
and $F''(x) < 0$ for $x\in (x_0, 1/7]$. Since $F'(0) = 0$ and $F'(1/7) < 0$, there exists $x_1 \in (x_0, 1/7)$ such that $F'(x_1) = 0$, $F'(x) > 0$ on $(0, x_1)$, and $F'(x) < 0$ on $(x_1, 1/7)$. Note that $F(0) = 0$ and $F(1/7) > 0$. Thus, $F(x) \ge 0$ on $[0, 1/7]$ and (\ref{Claim1ineq}) holds.
\end{proof}

Therefore, by (\ref{muinequito}) and (\ref{Claim1ineq}), we have
$$
\sum_{S: |S|=\lfloor q/2 \rfloor,\lceil q/2 \rceil} \beta_S \geq  \begin{cases} 
0.76 & \quad \text{if $q=5$}\\
0.80  & \quad \text{if $q\geq 7$}.
\end{cases},
$$
as long as $\bbeta \in \FEAS_q(\gamma)$ for $\gamma \leq 1/4$ sufficiently close to $1/4$, as desired.
\end{proof}

We are ready to embark on the proof of Theorem \ref{Opt5}. We use the following approach:

\begin{enumerate}[(i)]
\item Assume that $\bbeta \in \FEAS_q(\gamma)$ is a solution of \textbf{OPT} for odd $q\geq 5$. Then since the vector $\balpha$ defined in the statement of Theorem \ref{Opt5} is also a feasible vector, we must have $\OBJ_q(\bbeta)\geq \OBJ_q(\balpha)$.

\item By Lemma \ref{Claim2}, (\ref{betaineq0}) holds, which implies that $\supp_q(\bbeta)$ must have at least one set of size $\lceil q/2 \rceil$ or $\lfloor q/2 \rfloor$. Additionally, by Lemma \ref{JackTheRipper2}, $\supp_q(\bbeta)$ contains at most one set of size $\lceil q/2 \rceil$ and at most two sets of size $\lfloor q/2 \rfloor$.  Therefore, we divide our argument into the following disjoint cases. 

The support of $\bbeta$ contains

\textbf{Case 1:} no set of size $\lceil q/2 \rceil$ and exactly one set of size $\lfloor q/2 \rfloor$, or vice versa.

\textbf{Case 2:} no set of size $\lceil q/2 \rceil$ and exactly two sets of size $\lfloor q/2 \rfloor$.

\textbf{Case 3:} exactly one set $A$ of size $\lceil q/2 \rceil$, a set $B$ of size $\lfloor q/2 \rfloor$, but not $A^c$.

\textbf{Case 4:} exactly one set $A$ of size $\lceil q/2 \rceil$ and $A^c$, and $|\supp_q(\bbeta)|>2$.

\textbf{Case 5:} exactly one set $A$ of size $\lceil q/2 \rceil$ and $A^c$ only.
\end{enumerate}

We show that cases 1, 2, 3, and 4 are impossible, and conclude that $\bbeta$ must fall into Case 5. Then we show that $\bbeta$ must be of the same form as $\balpha$ in Theorem \ref{Opt5}.

\begin{proof}[Proof of Theorem \ref{Opt5}]
Throughout the proof we will be making a series of claims which hold for all $\gamma$ sufficiently close to $1/4$. Suppose that $\bbeta \in \FEAS_q(\gamma)$ is a solution of \textbf{OPT} for odd $q\geq 5$. Then
\begin{equation}  \label{mainineq1}
\OBJ_q(\bbeta) \geq \OBJ_q(\balpha),
\end{equation}
where $\balpha$ is a vector as defined in the statement of Theorem \ref{Opt5}, and by Lemma \ref{Claim2},
\begin{equation}\label{betaineq}
\sum_{S: |S|=\lfloor q/2 \rfloor,\lceil q/2 \rceil} \beta_S \geq \psi(q)= \begin{cases}
                								          	          0.76 & \text{ if } q=5\\
											                  0.80 & \text{ if } q\geq 7,
                                                                                                        \end{cases}
\end{equation}
Let us recall our constraints for feasible vectors. We have
\begin{equation} \label{Vconstraint0}
\V_q(\bbeta)=\sum_S \beta_S=1
\end{equation}
and
\begin{equation} \label{Econstraint0}
1-2\gamma \geq 1-2\E_q(\bbeta) = \sum_{(B,S): B\cap S\neq \emptyset} \beta_B\beta_S
\end{equation}
by (\ref{Econstraint}).

As was explained above, we proceed with our five cases.




\noindent \textbf{Case 1:} \textit{The support of $\bbeta$ contains no set of size $\lceil q/2 \rceil$ and exactly one set of size $\lfloor q/2 \rfloor$, or vice versa.}

Let $S$ be the only set of size $\lfloor q/2 \rfloor$ in $\supp_q(\bbeta)$ and suppose that no sets of size $\lceil q/2 \rceil$ are present in $\supp_q(\bbeta)$. Then by (\ref{betaineq}), $\beta_S \geq \psi(q)$, and using (\ref{Econstraint0}), we have
$$ \psi(q)^2 \leq \beta_S^2 \leq 1-2\E_q(\bbeta) \leq 1-2\gamma.$$
However, as $\gamma \to 0.25^-$, note that $\psi(5)^2\to 0.76^2\geq 0.57$, $\psi(q)^2\to 0.8^2= 0.64$ for $q\geq 7$, and $1-2\gamma \to 0.5$. Therefore, the inequality above cannot hold for all $\gamma$ sufficiently close to $1/4$. Thus, $\bbeta \notin \FEAS_q(\gamma)$, which is a contradiction.

Exactly the same argument holds if $S$ is the only set of size $\lceil q/2 \rceil$ in $\supp_q(\bbeta)$ and there are no sets of size $\lfloor q/2 \rfloor$ in $\supp_q(\bbeta)$. In either subcase we obtain a contradiction and therefore, this entire case is impossible.

For Cases 2, 3, and 4 we apply the results from \cite{Ma} that were stated in Section 3.

\noindent \textbf{Case 2:} \textit{The support of $\bbeta$ contains no set of size $\lceil q/2 \rceil$ and exactly two sets of size $\lfloor q/2 \rfloor$.}

Let $S_1$ and $S_2$ be the only two sets of size $\lfloor q/2 \rfloor$ in $\supp_q(\bbeta)$. By Lemma \ref{JackTheRipper2}, we have $S_1\cap S_2 =\emptyset$. Then Theorem \ref{JackTheRipper2} implies that we have only two possibilities for $\supp_q(\bbeta)$, which we consider below.

\noindent \textbf{Subcase 2.1:} $\supp_q(\bbeta)=\{S_1, S_2, (S_1\cup S_2)^c\}$.\\

Since $\beta_{S_1}+\beta_{S_2}<1$, we have
\begin{align*}
\OBJ_q(\bbeta) &= \log\paren{\lfloor q/2 \rfloor}(\beta_{S_1}+\beta_{S_2})\\
                      & < \frac{1}{2}\log\paren{\lceil q/2 \rceil \cdot \lfloor q/2 \rfloor}+\frac{\sqrt{1-4\gamma}}{2}\log\paren{\frac{\lceil q/2 \rceil}{\lfloor q/2 \rfloor}} \quad \text{for all $q\geq 5$}\\
                      &=\OBJ_q(\balpha).
\end{align*}
However, this contradicts (\ref{mainineq1}), and thus, this case is impossible.

\noindent \textbf{Subcase 2.2:} $\supp_q(\bbeta)=\{S_1, S_2, S_1\cup S_2, (S_1\cup S_2)^c\}$.\\

In this case, the entries of $\bbeta$ are subject to the following constraints
\begin{equation}\label{subcase2.2cons1}
\beta_{S_1}+\beta_{S_2}+\beta_{S_1\cup S_2}+\beta_{(S_1\cup S_2)^c}=1,
\end{equation}
\begin{equation}\label{subcase2.2cons2}
(\beta_{S_1\cup S_2}+\beta_{S_1}+\beta_{S_2})^2-2\beta_{S_1}\beta_{S_2} < 1-2\gamma,
\end{equation}
\begin{equation}\label{subcase2.2cons3}
\beta_{S_1}+\beta_{S_2} \geq \psi(q),
\end{equation}
where $\beta_{S_1} >0$, $\beta_{S_2} > 0$, $\beta_{S_1\cup S_2}> 0$, and $\beta_{(S_1\cup S_2)^c}> 0$. Solving (\ref{subcase2.2cons2}) for $\beta_{S_1\cup S_2}$ we obtain
$$
\beta_{S_1\cup S_2} \leq \sqrt{1-2\gamma+2\beta_{S_1}\beta_{S_2}}-(\beta_{S_1}+\beta_{S_2}),
$$
noting that $1-2\gamma +2\beta_{S_1}\beta_{S_2}\geq 0$ for any $\gamma \leq 1/4$. Let $t=\beta_{S_1}+\beta_{S_2}$. Then by (\ref{subcase2.2cons1}) and (\ref{subcase2.2cons2}), we have $\psi(q) \leq t \leq 1$. Since $2\beta_{S_1}\beta_{S_2} \leq t^2/2$, we have
$$ 
\beta_{S_1\cup S_2} \leq \sqrt{1-2\gamma+t^2/2}-t.
$$
Thus,
\begin{align*}
\OBJ_q(\bbeta) &=\log\paren{\lfloor q/2 \rfloor}t+\log(q-1)\beta_{S_1\cup S_2}\\
		       &\leq \log\paren{\frac{q-1}{2}}t+\log(q-1)(\sqrt{1-2\gamma+t^2/2}-t)\\	
		       &\leq -\log(2)t+\log(q-1)\sqrt{1-2\gamma+t^2/2}=: \lambda(t,q,\gamma).
\end{align*}
We will maximize $\lambda(t,q,\gamma)$ with respect to $t$. Note that $2t^2-8\gamma+4\geq 0$ for $\gamma \leq 1/4$, and hence,
$$
\frac{d\lambda}{dt}=-\log(2)+\frac{\log(q-1) \, t}{\sqrt{2t^2-8\gamma+4}} \geq 0
$$
when
$$ 
t \geq \sqrt{\frac{\log(2)\, (4-8\gamma)}{\log(q-1)^2-2\log(2)^2}}=:\tau(q, \gamma).
$$

Let us first consider when $q=5$. Note that 
$$
\tau(5, \gamma)>0.76=\psi(5)
$$
for all $\gamma \leq 1/4$. Thus, $\lambda(t,5, \gamma)$ is decreasing with respect to $t$ on $[\psi(5), \tau(5, \gamma))$ and is increasing with respect to $t$ on $[\tau(5, \gamma), 1]$. Thus,
$$ 
\OBJ_5(\bbeta)\leq \max \{\lambda(\psi(5), 5, \gamma),  \lambda(1, 5, \gamma) \}.
$$
Note that
$$
\lim_{\gamma \to \frac{1}{4}^-} \lambda(\psi(5), 5, \gamma)= \lambda(\psi(5), 5, 1/4)=-\log(2)\cdot 0.76+\log(4)\sqrt{1/2+(0.76)^2/2}\approx 0.70
$$
and
$$
\lim_{\gamma \to \frac{1}{4}^-} \lambda(1, 5, \gamma)= \lambda(1, 5, 1/4)=\log(2)\approx 0.69.
$$
Since $\lambda(t,5,\gamma)$ is continuous with respect to $\gamma$, then 
\[
\OBJ_5(\bbeta)\leq  \lambda(\psi(5), 5, \gamma) =-\log(2)\cdot 0.76+\log(4)\sqrt{1-2\gamma+(0.76)^2/2}
\]
for $\gamma$ sufficiently close to $1/4$. Since 
$$
\lim_{\gamma \to \frac{1}{4}^-} \frac{1}{2}\log(6)+\frac{\sqrt{1-4\gamma}}{2}\log\paren{\frac{3}{2}} =\frac{1}{2}\log(6) \approx 0.89,
$$ 
then 
\[
\OBJ_5(\bbeta)  \leq  \lambda(\psi(5), 5, \gamma) < \frac{1}{2}\log(6)+\frac{\sqrt{1-4\gamma}}{2}\log\paren{\frac{3}{2}} =\OBJ_5(\balpha),
\]
for $\gamma$ sufficiently close to $1/4$. We have contradicted (\ref{mainineq1}), and hence, this subcase is impossible for $q=5$.

Now suppose that $q\geq 7$. Note that
$$
\lim_{\gamma \to \frac{1}{4}^-} \tau(q, \gamma)=\tau(q,1/4)= \sqrt{\frac{2\log(2)}{\log(q-1)^2-2\log(2)^2}} \leq 0.80=\psi(q)
$$
for all $q\geq 7$. Since $\tau(q, \gamma)$ is decreasing with respect to $\gamma$ as $\gamma \to 0.25^-$, then 
$$
\tau(q, \gamma) \leq \psi(q)
$$ 
for all $\gamma \leq 1/4$ and $q\geq 7$. Thus, $\lambda(t,q, \gamma)$ is increasing on the interval $[\psi(q), 1]$ with respect to $t$, and hence,
$$
\OBJ_q(\bbeta) \leq \lambda(1,q,\gamma).
$$
Notice that
$$
 \lim_{\gamma \to \frac{1}{4}^-} \lambda(1,q,\gamma)=\lambda(1,q,1/4)= \frac{\log(q-1)}{2}<\frac{1}{2}\log\paren{\lceil q/2 \rceil \cdot \lfloor q/2 \rfloor},
$$
where the last inequality holds for all $q\geq 7$. Thus, for $\gamma$ sufficiently close to $1/4$
\[
\OBJ_q(\bbeta) < \frac{1}{2}\log\paren{\lceil q/2 \rceil \cdot \lfloor q/2 \rfloor}+\frac{\sqrt{1-4\gamma}}{2}\log\paren{\frac{\lceil q/2 \rceil}{\lfloor q/2 \rfloor}}=\OBJ_q(\balpha).
\]
However, this contradicts (\ref{mainineq1}), and thus, this subcase is also impossible for $q\geq 7$.

\noindent \textbf{Case 3:} \textit{The support of $\bbeta$ contains exactly one set $A$ of size $\lceil q/2 \rceil$, a set $B$ of size $\lfloor q/2 \rfloor$, but not $A^c$.}

Since $A^c \notin \supp_q(\bbeta)$, then $A\cap B \neq \emptyset$. If $B$ is the only set of size $\lfloor q/2 \rfloor$ in $\supp_q(\bbeta)$, then $\beta_A+\beta_B \geq \psi(q)$, and hence, by (\ref{Econstraint0}), we have
$$ \psi(q)^2 \leq (\beta_A+\beta_B)^2 \leq 1-2\E_q(\bbeta) \leq 1-2\gamma.$$
However, as was shown in Case 1, the inequality above cannot hold for all $\gamma$ sufficiently close to $1/4$. Thus, $\bbeta \notin \FEAS_q(\gamma)$, a contradiction.

We may assume that there exists another set $C\in\supp_q(\bbeta)\setminus \{B\}$ that satisfies $|C|=\lfloor q/2 \rfloor$. By Lemma \ref{JackTheRipper2},  $B\cap C=\emptyset$ and $B$ and $C$ are the only sets of size $\lfloor q/2 \rfloor$ in $\supp_q(\bbeta)$. Since $B$ and $C$ are disjoint, but they each intersect with the set $A$, Theorem \ref{JackTheRipper2} also implies that $A=B\cup C$ and that
$$ \supp_q(\bbeta)=\{ A, B, C, (B\cup C)^c\}.$$
However, then we would have $(q+1)/2=|A|=|B|+|C|=q-1,$ which is impossible. Therefore, this entire case is impossible.

\noindent \textbf{Case 4:} \textit{The support of $\bbeta$ contains exactly one set $A$ of size $\lceil q/2 \rceil$ and $A^c$, and $|\supp_q(\bbeta)|>2$.}

 Since $A, A^c \in \supp_q(\bbeta)$ and $|\supp_q(\bbeta)|>2$, then any other set in $\supp_q(\bbeta)$ must intersect $A$ or $A^c$. Recall, by Observation \ref{obs}, we have
$$
|\supp_q(\bbeta)|> \displaystyle \lceil 1/(1-2\gamma) \rceil +1=3,
$$ 
since $\lceil 1/(1-2\gamma)\rceil=2$ for $\gamma$ sufficiently close to $1/4$. One set in $\supp_q(\bbeta)$ must be the union of precisely two other set in $\supp_q(\bbeta)$ by Theorem \ref{JackTheRipper2}. Thus, the only possibility is that
$$ \supp_q(\bbeta)=\{A, A^c, A_1, A_2\}, $$
where $A_1$ and $A_2$ are disjoint sets which form a $2$-partition of $A$ or $A^c$.\\

Without loss of generality, assume that $A_1$ and $A_2$ form a $2$-partition of $A$. Then the entries of $\bbeta$ are subject to the following constraints
\begin{equation}\label{case4cons1}
\beta_{A}+\beta_{A^c}+\beta_{A_1}+\beta_{A_2}=1,
\end{equation}
\begin{equation}\label{case4cons2}
\beta_A^2+\beta_{A^c}^2+\beta_{A_1}^2+\beta_{A_2}^2+2\beta_A(\beta_{A_1}+\beta_{A_2}) \leq 1-2\gamma,
\end{equation}
\begin{equation}\label{case4cons3}
\beta_{A}+\beta_{A^c} \geq \psi(q),
\end{equation}
where $\beta_{A} > 0$, $\beta_{A^c} > 0$, $\beta_{A_1}>0$, and $\beta_{A_2}> 0$. We can combine (\ref{case4cons2}) and (\ref{case4cons2}) by substituting  $\beta_{A^c}=1-\beta_A-\beta_{A_1}-\beta_{A_2}$ into (\ref{case4cons2}) to obtain
\begin{align} \label{case4cons2b}
1-2\gamma &\geq \beta_A^2+\beta_{A^c}^2+\beta_{A_1}^2+\beta_{A_2}^2+2\beta_A(\beta_{A_1}+\beta_{A_2}) \nonumber \\
                 & =\frac{1}{2}(2\beta_A+2\beta_{A_1}+2\beta_{A_2}-1)^2+\frac{1}{2}-\beta_{A_1}^2-\beta_{A_2}^2-2\beta_{A_1}\beta_{A_2}.
\end{align}
By solving for $\beta_A$ in (\ref{case4cons2b}) we obtain
$$
\beta_{A} \leq \frac{\sqrt{1-4\gamma+2(\beta_{A_1}^2+\beta_{A_2}^2+2\beta_{A_1}\beta_{A_2})}+1}{2}-(\beta_{A_1}+\beta_{A_2}).
$$
Let $t=\beta_{A_1}+\beta_{A_2}$. Then $0 < t \leq 1-\psi(q)$ by (\ref{case4cons1}) and (\ref{case4cons3}). Note that 
$$
\beta_{A_1}^2+\beta_{A_2}^2+2\beta_{A_1}\beta_{A_2} \leq t^2,
$$
and hence, we have
$$ 
\beta_{A} \leq \frac{\sqrt{1-4\gamma+2t^2}+1}{2}-t.
$$
Thus,
\begin{align*}
\OBJ_q(\bbeta) &\leq \log\paren{\lfloor q/2 \rfloor}(\beta_{A^c}+\beta_{A_1}+\beta_{A_2})+\log\paren{\lceil q/2 \rceil}\beta_{A}\\
		       &=\log\paren{\frac{q-1}{2}}+\log\paren{\frac{\lceil q/2 \rceil}{\lfloor q/2 \rfloor}}\beta_{A}, \quad \text{ by (\ref{case4cons1})} \\
		       & \leq\log\paren{\frac{q-1}{2}}+\log\paren{\frac{\lceil q/2 \rceil}{\lfloor q/2 \rfloor}}\paren{\frac{\sqrt{1-4\gamma+2t^2}+1}{2}-t}=: \rho(t,q,\gamma).
\end{align*}
We will maximize $\rho(t,q, \gamma)$ with respect to $t$. Note that $\sqrt{1-4\gamma+2t^2}\geq 0$ for $\gamma \leq 1/4$, and hence,
$$
\frac{d\rho}{dt}=  \frac{t}{\sqrt{1-4\gamma+2t^2}}+1 < 0
$$
for $t >\sqrt{4\gamma -1}$.
Since $4\gamma -1\leq 0<t$, then $\rho(t,q,\gamma)$ is always decreasing with respect to $t$ on $[0, 1-\psi(q)]$. Thus,
\[ 
\OBJ_q(\bbeta) < \rho(0,q,\gamma)=\frac{1}{2}\log\paren{\lceil q/2 \rceil \cdot \lfloor q/2 \rfloor}+\frac{\sqrt{1-4\gamma}}{2}\log\paren{\frac{\lceil q/2 \rceil}{\lfloor q/2 \rfloor}}=\OBJ_q(\balpha),
\]
which contradicts (\ref{mainineq1}). Therefore, this case is impossible.

\noindent \textbf{Case 5:} \textit{The support of $\bbeta$ contains exactly one set $A$ of size $\lceil q/2 \rceil$ and $A^c$ only.}

For any $\gamma \leq 1/4$, define the roots of the equation $x(1-x)= \gamma$ by
\[ 
M(\gamma)^+=\frac{1+\sqrt{1-4\gamma}}{2} \quad \text{ and } \quad  M(\gamma)^-=\frac{1-\sqrt{1-4\gamma}}{2}. 
\]

In this case, we have $\supp_q(\bbeta)=\{A, A^c\}$. Then by (\ref{Vconstraint0}), we have
$$ \beta_A+\beta_{A^c}=1,$$
and we also have
$$ \E_q(\bbeta)=\beta_A\beta_{A^c}= \beta_A(1-\beta_A) \geq\gamma,$$
which implies that $M(\gamma)^-\leq \beta_A \leq M(\gamma)^+$. Similarly, we can also show that $M(\gamma)^-\leq \beta_{A^c} \leq M(\gamma)^+$.
Therefore, it is clear that in this case, we have
\[
\OBJ_q(\bbeta)\leq \log\paren{\frac{q-1}{2}} M(\gamma)^-+\log\paren{\frac{q+1}{2}} M(\gamma)^+=\OBJ_q(\balpha),\]
where equality holds everywhere if and only if $\beta_{A^c}=M(\gamma)^-$ and $\beta_A=M(\gamma)^+$. That is, equality holds if and only if $\bbeta$ is of the same form as $\balpha$ in the statement of Theorem \ref{Opt5}.

We have considered all possible cases for the solution vector $\bbeta$ and have shown that the only possibility is that $\bbeta$ must be of the same form as the vector $\balpha$ defined in the statement of Theorem \ref{Opt5}. Our proof is complete.
\end{proof}

\begin{corollary}\label{G_graph_is_Turan}
Let $n$ be a positive integer. If $\balpha$ solves $\OPT_q(1/4)$ for $q \geq 5$, then the graph $G_{\balpha}(n)$ is isomorphic to $T_2(n)$.
\end{corollary}

\begin{proof}
If $\balpha$ is a solution of $\OPT_q(1/4)$ then according to Theorem \ref{Opt5}, $\supp_q(\balpha)=\set{A, A^c}$, where $A\subseteq [q]$ such that $\abs{A}=\lceil q/2 \rceil$, and $\alpha_A=\alpha_{A^c}=1/2$. Then by Construction $\bm{G_{\balpha}(n)}$, the graph $G_{\balpha}(n)$ is a complete $2$-partite graph with parts $V_A$ and $V_{A^c}$, where $\abs{V_A}=\lfloor n/2\rfloor$ and $\abs{V_{A^c}}=\lceil n/2 \rceil$, or vice versa. Therefore, $G_{\balpha}(n) \cong K_{\lfloor n/2\rfloor,\lceil n/2 \rceil}$, which is isomorphic to $T_2(n)$.  
\end{proof}


\section{Approximate Version of Theorem 1}\label{four}
This section is dedicated to proving an ``approximate" version of Theorem \ref{main_theorem}. This version is nearly the same as Theorem \ref{main_theorem}, but has an additional requirement: namely, that a $(n, t_2(n))$-graph $G$ must be $\delta n^2$-close to $T_2(n)$ for sufficiently small $\delta>0$. That is, $T_2(n)$ ``locally maximizes" the number of $q$-colorings among the class of $(n, t_2(n))$-graphs for odd $q\geq 5$.

The main result of this section is as follows.

\begin{theorem}\label{local_theorem}
There exists a $\delta>0$ such that the following holds for sufficiently large $n$. Let $q\geq 2$ be an odd integer and let $G$ be a $(n, t_2(n))$-graph such that $G$ is $\delta n^2$-close to $T_2(n)$. Then $G$ has at most as many $q$-colorings as $T_2(n)$, with equality holding if and only if $G$ is isomorphic to $T_2(n)$.  
\end{theorem}

Intuitively, Theorem \ref{local_theorem} states that if a graph is ``close", with respect to edit distance, in structure to $T_2(n)$ then the number of its $q$-colorings is at most $P_{T_2(n)}(q)$ for odd $q \geq 5$.

The following four results will be referenced throughout the proof of Theorem \ref{local_theorem}. The first result, Lemma \ref{partitionlemma}, states the existence of a particular partition of the vertex set of a $(n, t_2(n))$-graph that is $\delta n^2$-close $T_2(n)$.


\begin{lemma} \label{partitionlemma}
Let $\delta>0$. If $G$ is a $(n, t_2(n))$-graph that is $\delta n^2$-close to $T_2(n)$ then there exists a partition $A_1\cupdot A_2=V(G)$ such that $e(G[A_1, A_2])\geq t_2(n)-\frac{\delta}{2}n^2$.
\end{lemma}
\begin{proof}
Suppose that $T_2(n)$ has vertex set partition $A_1\cupdot A_2=V(T_2(n))$. Since $G$ is $\delta n^2$-close to $T_2(n)$, then by definition, $V(G)=V(T_2(n))$. We claim that $$e(G[A_1, A_2])\geq t_2(n)-\frac{\delta}{2}n^2.$$

First observe that
\begin{equation} \label{loclem_1_ineq}
E(G[A_1, A_2])=E(G)\cap E(T_2(n))=(E(G)\cup E(T_2(n)))\setminus (E(G)\triangle E(T_2(n))),
\end{equation}
where the first equality holds since $T_2(n)$ has all possible edges between the parts $A_1$ and $A_2$. By taking the cardinalities of each set in (\ref{loclem_1_ineq}) we have
\[
e(G[A_1, A_2]) \geq 2t_2(n)-e(G[A_1,A_2])-\delta n^2.
\]
Thus, we have our desired result.
\end{proof}

We will need the following upper bound on $P_G(q)$ from \cite{FL1} later. 
\begin{lemma}[\cite{FL1}] \label{FLBound}
Let $G$ be an $(n,m)$-graph and let $q \geq 2$ be an integer. Then
$$ P_G(q) \leq \paren{1-\frac{1}{q}}^{\lceil (\sqrt{1+8m}-1)/2} q^n \leq \paren{1-\frac{1}{q}}^{\lceil (\sqrt{m}-1)/2} q^n.  $$
\end{lemma}

We will now embark on the proof of Theorem \ref{local_theorem}. Our approach is very similar to the one used by Norine \cite{Norine} in his proof of Lemma 4.2, with some adjustments, since $q$ is not an even integer.

\begin{proof}[Proof of Theorem \ref{local_theorem}]
Throughout the proof we will be making a series of claims which hold for positive $\delta$, sufficiently small as a function of $q$, and for $n$, sufficiently large as a function of $q$ and $\delta$. The eventual choice of $\delta$ and $n$ will be implicitly made so that all of our claims are valid.

Let $G$ be a $(n, t_2(n))$-graph, where $n \geq 2$. Suppose, to the contrary, that $G$ has more $q$-colorings than $T_2(n)$. By Lemma  \ref{partitionlemma}, there exists a partition $A_1\cupdot A_2=V(G)$ such that $e(G[A_1,A_2])\geq t_2(n)-\frac{\delta}{2}n^2$. Assume that the pair $(A_1, A_2)$ is chosen to maximize $e(G[A_1,A_2])$. Let $\delta ':=2n^{-1} \abs{|A_1|-\frac{n}{2}}=2n^{-1} \abs{|A_2|-\frac{n}{2}}$. Then
\[
e(G[A_1, A_2])  \leq \paren{\frac{n}{2}}^2-\paren{\frac{n\delta'}{2}}^2 \leq t_2(n)+\frac{1}{4}-\frac{1}{2}\paren{\frac{n\delta'}{2}}^2.
\]
Then for sufficiently small $\delta$ and sufficiently large $n$, $(\delta'n/2)^2\leq 1/2+\delta n^2\leq (q/2)\delta n^2$, and thus, $\delta' \leq \sqrt{(q/2)\delta}$. Therefore, it suffices to show that the conclusion of the lemma holds as long as not only $\delta$, but $\max\set{\delta, \delta'}$ is sufficiently small. To simplify the notation at the expense of overloading it, we will use $\delta$ in the remainder of the proof to denote $\max\set{\delta, \delta'}$. In particular, we have $\abs{|A_i|-n/2}\leq (\delta n) / 2$ for all $i \in [2]$ for sufficiently small $\delta$.

Let $\epsilon:=\sqrt{\delta}$. We say that a vertex $v \in V(G)$ is \textit{good} if $d_{A_i}(v)\geq (1-\epsilon)\abs{A_i}$ for the $i \in [2]$ such that $v \notin A_i$; that is, $v$ has ``many" neighbors in the part $A_i$ that does not contain $v$. Otherwise, we say that $v$ is \textit{bad}. Let $B$ denote the set of bad vertices of $G$. By counting the edges in $\overline{G}[A_1, A_2]$, where $\overline{G}$ denotes the complementary graph of $G$, we obtain
\[
\epsilon \paren{1-\delta}\frac{n}{2} \abs{B}  \leq e(\overline{G}[A_1,A_2]) =t_2(n)-e(G[A_1,A_2])\leq \frac{\delta}{2}n^2,
\]
and hence, $\abs{B} \leq \frac{\epsilon}{1-2\epsilon^2}n\leq 2\epsilon n$ for sufficiently small $\epsilon$ (and hence, $\delta$).

Let $f:V(G)\to [q]$ be a $q$-colorings of $G$. For each $i \in [2]$ define 
$$ \R_f(i):=\set{c \in [q]: \abs{f^{-1}(c)\cap A_i}> \epsilon \abs{A_i}}, $$
i.e., $\R_f(i)$ is the set of colors which occur relatively frequently in $A_i$ under the coloring $f$. We say that each color in $\R_f(i)$ is an \textit{essential color in $A_i$}.

We make two observations about $\R_f(i)$. First, note that we can ensure that $\R_f(i) \neq \emptyset$ for each $i \in [2]$ by ensuring that $\abs{A_i} = \sum_{c \in [q]} \abs{f^{-1}(c)\cap A_i}>q\epsilon \abs{A_i}$ by choosing $\epsilon <1/q$. Secondly, the sets $\R_f(1)$ and $\R_f(2)$ are disjoint. Note that for every essential color $c \in \R_f(i)$ we have $f^{-1}(c)\subseteq A_i \cup B$. Otherwise, if say, $i=1$, and there were some $v \in A_2 \setminus B$ such that $f(v)=c$, then since $d_{A_1}(v)\geq (1-\epsilon)\abs{A_1}$, the vertex $v$ must be adjacent to some vertex in $A_1$ with color $c$, contradicting that $f$ is a $q$-coloring. Therefore, we see that there exists at least one essential color in each $A_i$, and that $A_1$ and $A_2$ cannot share essential colors.

Let us define the vector $\Rbold_f:=(\R_f(1), \R_f(2))$. Given another vector $\Rbold=(\R_1, \R_2)$ such that the components $\R_1$ and $\R_2$ are disjoint nonempty subsets of $[q]$, let
$$ \P_G(\Rbold):=\abs{\set{f:V(G)\to [q] : \R_i \text{ is the set of essential colors in } A_i \text{ for each } i \in[2]}}. $$
We will bound $\P_G(\Rbold)$ in two distinct cases, each of which is based upon a comparison of the cardinalities of the sets $\R_1$ and $\R_2$.


\noindent \textbf{Case 1:} The components of $\Rbold=(\R_1, \R_2)$ that satisfy $$(\abs{\R_1}, \abs{\R_2}) \notin \set{\paren{\lfloor q/2 \rfloor,\lceil q/2 \rceil}, \paren{\lceil q/2 \rceil, \lfloor q/2 \rfloor}}.$$

We can estimate $\P_G(\Rbold)$ in the following way: (i) Allow the vertices of $B$ to be colored arbitrarily, (ii) allow $\abs{R_i}$ choices of colors for each of the vertices in $A_i$, (iii) account for the number of subsets of $A_i$ which will \textit{not} be colored with any of the $\abs{\R_i}$ essential colors, and (iv) color the subset of $A_i$ chosen in (iii). By estimating $\P_G(\Rbold)$ this way we obtain
\begin{align*}
\P_G(\Rbold) & \leq q^{\abs{B}} \paren{\prod_{i=1}^2 \abs{\R_i}^{\abs{A_i}}\cdot 2 \binom{\abs{A_i}}{(q-\abs{\R_i})\epsilon \abs{A_i}}\cdot (q-\abs{\R_i})^{(q-\abs{\R_i})\epsilon \abs{A_i}}}\\
				   &\leq 4\cdot q^{2\epsilon n} \cdot \paren{(\lceil q/2 \rceil+1)(\lfloor q/2 \rfloor -1)}^{(1+\delta)\frac{n}{2}} \paren{e/\epsilon}^{q\epsilon n}\\	   
				   &= 4 \paren{(\lceil q/2 \rceil+1)(\lfloor q/2 \rfloor -1)}^{n/2} \cdot \exp\paren{\paren{\frac{\delta}{2} \log\paren{\frac{q^2-9}{4}}+2\epsilon \log(q)+q\epsilon \log\paren{e/\epsilon}}n}\\
				   &<\frac{1}{3^q}\paren{\lceil q/2 \rceil \cdot \lfloor q/2 \rfloor}^{(n-2)/2},
\end{align*}
for $\epsilon$ (and hence, $\delta$) sufficiently small and $n$ sufficiently large, since 
$$\exp\paren{\paren{\frac{\delta}{2} \log\paren{\frac{q^2-9}{4}}+2\epsilon \log(q)+q\epsilon \log\paren{e/\epsilon}}n} \to 1
$$
as $\epsilon=\sqrt{\delta} \to 0$. It follows that
\begin{equation} \label{PGBound1}
\sum_{\Rbold} \P_G(\Rbold) < \sum_{\Rbold} \frac{1}{3^q}\paren{\lceil q/2 \rceil \cdot \lfloor q/2 \rfloor}^{(n-2)/2}=\paren{\lceil q/2 \rceil \cdot \lfloor q/2 \rfloor}^{(n-2)/2},
\end{equation}
where the summation is taken over all $\Rbold=(\R_1, \R_2)$ such that $$(\abs{\R_1}, \abs{\R_2}) \notin \set{(\lfloor q/2 \rfloor,\lceil q/2 \rceil), (\lceil q/2 \rceil, \lfloor q/2 \rfloor)}.$$ 

\noindent \textbf{Case 2:} The components of $\Rbold=(\R_1, \R_2)$ that satisfy $$(\abs{\R_1}, \abs{\R_2}) \in \set{\paren{\lfloor q/2 \rfloor,\lceil q/2 \rceil}, \paren{\lceil q/2 \rceil, \lfloor q/2 \rfloor}}.$$

In this case, we will bound $\P_G(\Rbold)$ when $\Rbold$ corresponds to a partition of $[q]$ into two parts, one of which is roughly of size $q/2$. Note that under any such $q$-coloring $f$, all of the vertices in $A_i\setminus B$ are only colored with colors from $\R_f(i)$. Otherwise, if there is a $q$-coloring $f$ such that $\Rbold_f=\Rbold$ and we have, for example, a vertex $v \in A_1\setminus B$ such that $f(v)=c$ for some $c \in \R_f(2)$, then since $d_{A_2}(v)\geq (1-\epsilon) \abs{A_2}$ and there exist more than $\epsilon k$ vertices with color $c$ in $A_2$, then $v$ would be adjacent to some vertex of color $c$ in $A_2$, a contradiction.

Suppose first that there exists a vertex $v \in V(G)$ such that $d_{A_i}(v) \geq \delta^{2/5}\abs{A_i}$ for every $i \in [2]$. We can estimate $\P_G(\Rbold)$ in the following way: (i) Arbitrarily color the vertices of $B$ with any of the $q$ colors, (ii) if $f(v) \in \R_f(j)$ for some $j \in [2]$ then arbitrarily color the neighbors of $v$ in $A_j \setminus B$ with any of the available $\abs{\R_f(j)}-1$ colors in $\R_f(j)\setminus \set{f(v)}$, (iii) arbitrarily color the vertices in $A_i\setminus B$ which are \textit{not} neighbors of $v$ using any of the colors in $\R_f(j)$, and (iv) arbitrarily color the vertices in $A_{3-j}$ using any of the colors in $\R_f(3-j)$. There are $\abs{\R_f(3-j)}^{A_{3-j}}$ possibilities. By estimating $\P_G(\Rbold)$ this way we obtain
\begin{align*}
\P_G(\Rbold) &\leq 2\cdot q^{\abs{B}} (\abs{\R_f(j)}-1)^{d_{A_j}(v)} \abs{\R_f(j)}^{\abs{A_j}-d_{A_j}(v)} \abs{\R_f(3-j)}^{\abs{A_{3-j}}} \\
		     & \leq 2\cdot q^{2\epsilon n}\paren{\frac{\lceil q/2 \rceil-1}{\lceil q/2 \rceil}}^{\delta^{2/5}(1-\delta)n/2} \paren{\lfloor q/2 \rfloor \lceil q/2 \rceil}^{n/2+\delta n}\\
		     & < \frac{1}{2^q} \paren{\lfloor q/2 \rfloor \lceil q/2 \rceil}^{n/2},
\end{align*}
for sufficiently small $\delta$ and sufficiently large $n$. Combining this with the previous calculations from Case 1 we obtain $
\P_G(\Rbold) < 2\paren{\lfloor q/2 \rfloor \lceil q/2 \rceil}^{n/2}$, which is less than the number of $q$-colorings of $T_2(n)$, a contradiction. Therefore, a vertex $v$ as above does not exist. It follows from the choice of the partition $(A_1, A_2)$ that for every $i \in [2]$ the subgraph $G[A_i]$ of $G$ has maximum degree at most $\delta^{2/5}n$. Let $e_i:=e(G[A_i \setminus B])$ for each $i \in [2]$. Then
$$  
\sum_{i=1}^2 \paren{e_i+\delta^{2/5}n \abs{B\cap A_i}} \geq \sum_{i=1}^2 e(G[A_i]) =e(G)-e(G[A_1, A_2])\geq \epsilon \paren{1-\delta}\frac{n}{2} \abs{B}.
$$
It follows that $e_1+e_2  \geq \delta^{2/5}\abs{B}n$ for sufficiently small $\delta$. Using Lemma \ref{FLBound} we obtain 
\begin{align*}
&\P_G(\Rbold) \nonumber \\ 
		     & \leq q^{\abs{B}} \paren{\lfloor q/2 \rfloor \lceil q/2 \rceil}^{(1+\delta)\frac{n}{2}} \prod_{i=1}^2 \paren{\frac{\lceil q/2 \rceil -1}{\lceil q/2 \rceil}}^{\sqrt{e_i}} \nonumber\\
		     &= \paren{\lceil q/2 \rceil \cdot \lfloor q/2 \rfloor}^{(1+\delta)\frac{n}{2}} \exp\paren{\log(q)\abs{B}-\log\paren{\frac{\lceil q/2 \rceil}{\lfloor q/2 \rfloor}}\sqrt{\delta^{2/5}\abs{B}n}+\log\paren{\lceil q/2 \rceil \cdot \lfloor q/2 \rfloor}\frac{\delta}{2}n} \nonumber\\
		     &\leq \paren{\lceil q/2 \rceil \cdot \lfloor q/2 \rfloor}^{n/2}\exp\paren{\paren{\log(q)-10^{-1/2}\delta^{-1/20}\log\paren{\frac{\lceil q/2 \rceil}{\lfloor q/2 \rfloor}}}\abs{B}+\log\paren{\lceil q/2 \rceil \cdot \lfloor q/2 \rfloor}\frac{\delta}{2}n}. 
\end{align*}
If $\abs{B} \neq 0$ then $\P_G(\Rbold)$ once again becomes negligible compared to $\paren{\lceil q/2 \rceil \cdot \lfloor q/2 \rfloor}^{n/2}$, as $\delta$ approaches $0$. It follows that $\abs{B}=0$. It suffices to assume that $G[A_1, A_2]$ is a complete bipartite graph. Indeed, if this were not the case, then there exist nonadjacent vertices $v_1 \in A_1$ and $v_2 \in A_2$ and adjacent vertices $x$ and $y$, either both in $A_1$ or both in $A_2$. Let the graph $G':=G-xy+v_1v_2$. If $f$ is a $q$-coloring in Case 2, then $f(x) \neq f(y)$ and $f(v_1) \neq f(v_2)$. Thus, we can apply the coloring $f$ to the graph $G'$ to obtain a proper $q$-coloring of $G'$ that satisfies the conditions of Case 2. Therefore, $\P_G(\Rbold)\leq \P_{G'}(\Rbold)$. We may repeat this process until all possible edges between $A_1$ and $A_2$ are present. Then $G \cong K_{\abs{A_1}, \abs{A_2}}$, and since $t_2(n)=e(G)$ then $G$ must be isomorphic to $T_2(n)$. Therefore, we have shown that any $(n, t_2(n))$-graph $G$ which is $\delta n^2$ close to $T_2(n)$ has at most as many $q$-colorings as $T_2(n)$, with equality holding if and only if $G$ is isomorphic to $T_2(n)$, provided that $\delta$ is sufficiently small and $n$ is sufficiently large.
\end{proof}
%
%
%
%
%

\section{Proof of Theorem \ref{main_theorem}}\label{five}

We are ready to prove our main result, Theorem \ref{main_theorem}. The proof combines our solution of \textbf{OPT} for odd $q \geq 5$ and Theorem \ref{local_theorem}, and uses a similar approach to Norine's proof of his main result, Theorem 1.1, in \cite{Norine}. 

\begin{proposition}[\cite{Loh}] \label{Eequality}
If $\balpha \in \FEAS_q(\gamma)$ solves $\OPT_q(\gamma)$, then $\E_q(\balpha)=\gamma.$
\end{proposition}

\begin{proof}[Proof of Theorem \ref{main_theorem}]
We proceed by contradiction. Assume there exists an increasing sequence of positive integers $\set{n_i}_{i=1}^\infty$ and a sequence of graphs $\set{H_i}_{i=1}^\infty$ such that $H_i$ is a $(n_i, t_2(n_i))$-graph, $H_i$ is not isomorphic to $T_2(n_i)$, and $H_i$ has at least as many $q$-colorings as any other $(n_i, t_2(n_i))$-graph. Choose $\epsilon>0$ so that a real number $\gamma \in [1/4 -\epsilon, 1/4]$ and the conclusion of Theorem \ref{Opt5} holds. We apply Theorem \ref{LPS2} for $\kappa=1/4$ and a sequence of positive real numbers $\set{\delta_i}_{i=1}^\infty$ with $0<\delta_i\leq \epsilon$ and $\lim_{i\to \infty} \delta_i=0$. By possibly restricting $\set{n_i}_{i=1}^\infty$ to a subsequence, we obtain a sequence $\set{\balpha_i}_{i=1}^\infty$ such that $H_i$ is $\delta_i n_i^2$-close to the graph $G_{\balpha_i}(n_i)$, $\balpha_i$ solves $\OPT_q(\gamma_i)$ for some real number $\gamma_i$ such that $\gamma_i \in [1/4-\epsilon, 1/4]$, and $\lim_{i\to \infty}\gamma_i=1/4$. 

Since $\balpha_i \in \FEAS_q(\gamma_i)$, and $\FEAS_q(\gamma_i)$ is a compact set, we may further restrict our sequence $\set{\balpha_i}_{i=1}^\infty$ (and hence, the sequences $\set{n_i}$ and $\set{\gamma_i}$) by assuming that the $\balpha_i$'s converge in the $L^1$-norm to a vector $\balpha^*$ with $\E_q(\balpha^*)=\gamma^*$.
Then by Theorem \ref{Opt5}, for odd $q \geq 5$, we have
\begin{align*}
&\OBJ_q(\balpha^*)-\frac{1}{2}\log\paren{\lceil q/2 \rceil \cdot \lfloor q/2 \rfloor}\\
&=\lim_{i \to \infty} \left[\OBJ_q(\balpha_i)-\paren{\frac{1}{2}\log\paren{\lceil q/2 \rceil \cdot \lfloor q/2 \rfloor}+\frac{\sqrt{1-4\gamma_i}}{2}\log\paren{\frac{\lceil q/2 \rceil}{\lfloor q/2 \rfloor}}}\right]=0.
\end{align*}
Therefore,
$$
\OBJ_q(\balpha^*)=\frac{1}{2}\log\paren{\lceil q/2 \rceil \cdot \lfloor q/2 \rfloor}.
$$
Since $\balpha_i$ solves $\OPT_q(\gamma_i)$, by Proposition \ref{Eequality}, $\E_q(\balpha_i)=\gamma_i$. Since $\set{\balpha_i}_{i=1}^\infty$ converges to $\balpha^*$ in the $L^1$-norm, we have
$$
1/4=\lim_{i \to \infty} \gamma_i= \lim_{i \to \infty}\E_q(\balpha_i)=\E_q(\balpha^*)=\gamma^*.
$$ 
Then $\gamma^*=1/4$, and hence, by Theorem \ref{Opt5},   $\balpha^*$ solves $\OPT_q(1/4)$. Then Corollary \ref{G_graph_is_Turan} tells us that the graph $G_{\balpha^*}(n)=T_2(n)$ for every $n$. 

Let $\delta>0$ be chosen so that the conclusion of Theorem \ref{local_theorem} holds. By Proposition \ref{LPS1}, $G_{\balpha_i}(n_i)$ is $\delta n_i^2/2$-close to $G_{\balpha^*}(n_i)=T_2(n_i)$ for sufficiently large $i$, since $\balpha_i \to \balpha^*$. We may assume that $\delta_i \leq \delta/2$ for sufficiently large $i$ since $\delta_i \to 0$ as $i \to \infty$.
Consequently, as each $H_i$ is $\delta_i n_i^2$-close to $G_{\balpha_i}(n_i)$, then the edit distance between $H_i$ and $G_{\balpha^*}(n_i)=T_2(n_i)$ is at most
$$
\frac{\delta}{2}n_i^2+\delta_i n_i^2 \leq \frac{\delta}{2}n_i^2+\frac{\delta}{2}n_i^2 = \delta n_i^2.
$$
That is, $H_i$ is $\delta n_i^2$-close to $T_2(n_i)$ for sufficiently large $i$. However, this contradicts Theorem \ref{local_theorem}, finishing the proof of the theorem.
\end{proof}

%
%
%
%
\section{Concluding remarks and open problems}\label{six}
In Section \ref{one}, we presented Conjecture \ref{main_conj} and stated that although several cases of the conjecture had been solved for various ranges of $r$, $q$, and $n$, it was not true in general, as counterexamples were discovered in \cite{Ma}. Nonetheless, several cases of the conjecture remain open, one of them being the following conjecture.

\begin{conj} \label{main_conj6}
Let $r$ and $q$ be integers such that $2\leq r \leq 9$ and $r \leq q$. Then for all $n \geq r$, the Tur\'{a}n graph $T_r(n)$ has more $q$-colorings than any other graph with the same number of vertices and edges.
\end{conj}

It seems that it may be difficult to resolve Conjecture \ref{main_conj6} for all $n \geq r$. However, asymptotic versions (for $n$ sufficiently large), may be more attainable by finding solutions of \textbf{OPT} for $2\leq r \leq 9$, $q \geq r$, and positive $\gamma$ that satisfy $\gamma \leq (q-1)/(2q)$.

%
%
%
\section{Acknowledgements}\label{ack}

I would like to thank Felix Lazebnik, my former advisor, for his all of his help and encouragement on this work while I was a doctoral student at the University of Delaware.

\end{document}